\RequirePackage{fix-cm}

\documentclass[smallextended]{svjour3} 

\smartqed  
\usepackage{graphicx}

\usepackage[latin1]{inputenx}
\usepackage{amsfonts}
\usepackage{latexsym}
\usepackage{enumerate}
\usepackage{amsmath}
\usepackage{amsmath,amssymb,mathrsfs,yhmath}
\usepackage{tikz}
\usetikzlibrary{matrix}

\usepackage{subfig}

\usetikzlibrary{positioning}

\usepackage[pdftex,bookmarks=true]{hyperref}
\usepackage{url}

\newcommand{\myjoin}{\begin{picture}(20,6)\put(3,3){\line(1,0){13
}}\end{picture}}

\newcommand{\al}{\alpha}
\newcommand{\be}{\beta}

\usepackage{anysize}
\marginsize{3cm}{3cm}{3cm}{3cm}

\newcommand{\cP}{\mathcal{P}}
\newcommand{\cC}{\mathcal{C}}
\newcommand{\TT}{\mathbb{T}}

\begin{document}

\title{On Finite Dimensional Jacobian Algebras}

\author{Sonia Trepode       \and
        Yadira Valdivieso-D\'iaz 
}


\institute{S. Trepode \at
              Departamento de Matem\'atica, Facultad de Ciencias Exactas y Naturales, Funes 3350, Universidad Nacional de Mar del Plata, 7600 Mar del Plata, Argentina \\
              Tel.: +542234753150\\
              \email{strepode@mdp.edu.ar}           
           \and
           Y. Valdivieso-D\'iaz \at
               Departamento de Matem\'atica, Facultad de Ciencias Exactas y Naturales, Funes 3350, Universidad Nacional de Mar del Plata, 7600 Mar del Plata, Argentina\\
              Tel.: +542234753150\\
              \email{valdivieso@mdp.edu.ar}     
}

\maketitle
\date{Received: date / Accepted: date}
\begin{abstract}
We show that Jacobian algebras arising from a sphere with
$n$-punctures, with $n\geq5$, are finite dimensional algebras. We
consider also a family of cyclically oriented quivers and we prove that,
 for any primitive potential, the associated Jacobian algebra is
finite dimensional.
\keywords{Jacobian algebras \and closed surfaces \and cyclically oriented quivers}
\subclass{MSC 16G20 \and MSC 13F60 \and MSC 57M20 \and MSC 16T30}
\end{abstract}

\section{Introduction}

Let $k$ be an algebraically closed field.
A potential $W$ for a quiver $Q$ is, roughly speaking, a linear combination of
cyclic paths in the  complete path algebra $k\langle\langle Q\rangle\rangle$.
The Jacobian algebra $\cP(Q,W)$ associated to a quiver with a potential
$(Q,W)$ is  the quotient of the complete path algebra $k\langle\langle Q\rangle\rangle$ modulo the Jacobian
ideal $J(W)$. Here, $J(W)$ is the closure of the ideal of $k\langle\langle Q\rangle\rangle$
which is generated by the cyclic derivatives of $W$ with respect to the arrows
of  $Q$.

Quivers with potential were introduced in\cite{DWZ08} in order to construct additive
categorifications of cluster algebras with skew-symmetric exchange
matrix. For the just mentioned categorification it is crucial that
the potential for $Q$ be non-degenerate, i.e. that it can be mutated along
with the quiver arbitrarily, see \cite{DWZ08} for more details on quivers with
potentials.

In the same year, Fomin, Shapiro and Thurston gave  in \cite{FST08} a class of cluster algebras arising from ideal
triangulations of surfaces with marked points. More precisely, each triangulation $\TT$ of a surface with marked points
$(S,M)$ by tagged
arcs corresponds to a cluster and the corresponding exchange matrix is
conveniently coded into a quiver $Q(\TT)$. Later, a link between these papers was
established by Labardini-Fragoso in \cite{LF09}, he considered
surfaces with a non-empty boundary and he gave a potential $W(\TT)$
associated with an ideal triangulation $\TT$ such that its corresponding
Jacobian algebra is finite dimensional.

In the first part of this work, we study Jacobian algebras
associated with ideal triangulations of a sphere with $n\geq 5$ punctures. Our
main result is the following:

\begin{theorem}\label{TeoSphere}
Let $(S, M)$ be  a sphere with $n$-punctures, where
$n\geq 5$. For every ideal triangulation $\TT$ of $(S,
M)$, the Jacobian algebra $\cP(Q(\TT),W(\TT))$ is
finite dimensional.
\end{theorem}

The case of a sphere with $4$ punctures was studied by  Barot and Geiss (in \cite{BG09}, Section 5), the algebra associated to this surface is a tubular cluster algebra.

In the theory of cluster algebras, primitive potentials, which are a
linear combination of all the oriented chordless cycles in a quiver $Q$,
appear in many contexts,  for example in cluster tilted algebras of
Dynkin type (\cite{DWZ08}, Section 9). Also,  it follows from
\cite{BT10} that cluster tilted algebras with cyclically oriented
quivers have a primitive potential (see definition of a cluster
tilted algebra in Section \ref{4}).

In the proof of our main Theorem \ref{TeoSphere}, we use a
particular ideal triangulation $\TT$ of a sphere with $n$-punctures such that the
quiver associated to $\TT$ is cyclically oriented but its associated potential
is not primitive.

In the second part of this work, we give a class of  cyclically oriented quivers
such that any  primitive potential induces a finite dimensional
Jacobian algebra.

The paper is organized as follows: In Section \ref{2}, we recall
some definitions of quivers with potentials, path algebras, Jacobian
algebras and ideal triangulations of surfaces. In Section \ref{3},
we prove that every Jacobian algebra associated with an ideal
triangulation of a sphere with $n\geq 5$ punctures are finite dimensional.
Finally, in Section \ref{4}, we give a combinatorial description
of a quiver $Q$ such that any of its primitive potentials induce a
finite dimensional Jacobian algebra.

\begin{remark}
While we were finishing this manuscript, we became aware of the recent
paper \cite{Lad12}, where Ladkani showed that Jacobian algebras of surfaces with an empty boundary and arbitrary
particular genus are
finite dimensional algebras.
\end{remark}


\section{Preliminaries}\label{2}


\subsection{Quivers and potentials}
In this subsection, we fix notations for path algebras and
complete path algebras, and recall basic definitions of quivers with potential (cf.
\cite{DWZ08}).

Let $Q$ be a finite quiver and $k$ be a field. We denote by $R$ the $k$-vector space
$k^{Q_{0}}$, by $A$  the $k$-vector space $k^{Q_1}$ and, for each
nonnegative integer $d$ by $A^d$ the $R$-bimodule
$\underbrace{A\otimes_R\dots\otimes_R A}_d$ .


With this notation,  the \textit{path algebra} of $Q$ is the $k$-algebra defined as the (graded) tensor algebra

$$k\langle Q \rangle=\displaystyle{\bigoplus_{d=0}^{\infty}A^d}$$

and the \textit{complete path algebra} of $Q$ is the $k$-vector space defined by

$$k\langle\langle Q \rangle\rangle=\displaystyle{\prod_{d=0}^{\infty}A^d}.$$

Also, $k\langle\langle Q\rangle\rangle$ is a topological $k$-algebra
with the $\mathfrak{m}$-adic  topology, where $\mathfrak{m}$ is the
ideal $\displaystyle{\prod_{d=1}^{\infty}A^d}$.

\begin{remark}\label{topology}
An important topological property of $k\langle\langle Q\rangle\rangle$  with the $\mathfrak{m}$-adic topology, for this work, is the following:

A sequence $(x_n)_{n\in \mathbb N}$ of elements of $k\langle\langle Q\rangle\rangle$ converges if and only if for every $d\geq 0$,
the sequence $(x^{(d)}_n)_{n\in \mathbb N}$ does, and $$\lim_{n\to\infty} x_n=\sum_{d\geq 0}\lim_{n\to\infty}x^{(d)}_n,$$
where $x^{(d)}_n$ denotes the degree-$d$ component of $x_n$.
\end{remark}

Notice that the elements of $k\langle\langle Q \rangle\rangle$ are
(possibly infinite) $k$-linear combinations of paths
in $Q$.

Denote by $k\langle\langle Q\rangle\rangle_{\textrm{cyc}}$ the
$k$-subspace of $k\langle\langle Q\rangle\rangle$ whose element are
$k$-linear combinations of cycles in $Q$.

\begin{definition}\cite[Definition 3.1]{DWZ08}
\begin{itemize}
\item A potential $W$ is any element of  the $k$-subspace $k\langle\langle Q\rangle\rangle_{\textrm{cyc}}$.
\item For every arrow $a$ in $Q_1$, we define the \textit{cyclic derivative} $\partial_{a}$ as
the continuous $k$-linear map $$k\langle\langle Q
\rangle\rangle_{\textrm{cyc}} \to k\langle\langle Q\rangle\rangle$$
acting on paths by

\begin{equation*}
\partial_{a}(a_1\cdots a_d)=\sum_{k=1}^{d}\delta_{a a_k}a_{k+1}\cdots a_d a_1\cdots a_{k-1}
\end{equation*}
\item The Jacobian ideal $J(W)$ of a potential $W$ is the closure
of the ideal $$\textrm{I}(W)=\langle \partial_{a}(W) \mid a\in Q_1\rangle $$
in $k\langle\langle Q\rangle\rangle$.

\item The Jacobian algebra $\cP(Q,W)$ is the quotient $k\langle\langle Q\rangle\rangle / J(W)$.
\end{itemize}
\end{definition}

\subsection{Triangulations of surfaces}

In this subsection, we review some facts concerning  triangulations  of surfaces (cf. \cite{FST08}).

\begin{definition}\cite[Definition 2.1]{FST08} A \textit{bordered surface with marked points} is a pair $(S, M)$, where  $S$ is a connected oriented
2-dimensional Riemann surface with a (possibly empty) boundary and
$M$ is a finite and non-empty set of points in $S$,
called \textit{marked points}, such that there is at least one
marked point on each connected component of the boundary of $S$.

The set $P$ of marked points in the interior of $S$ are called \textit{punctures}.
\end{definition}

In this paper, we study spheres with $n$ punctures, $n \geq 5$, due to the  following definitions, in order to avoid surfaces that cannot be
triangulated or there is only one triangulation,  we need to exclude:

\begin{itemize}
\item spheres with one or two punctures;
\item unpunctured or once-punctured monogons;
\item unpunctured digons; and
\item unpunctured triangles
\end{itemize}

\begin{definition}\cite[Definition 2.2 and 2.4]{FST08}
A (simple) \textit{arc} $\gamma$  in $(S, M)$ is a curve in $S$ such that:
\begin{itemize}
\item the endpoints of $\gamma$ are marked points in $M$;
\item $\gamma$ does not intersect itself, except that its endpoints may coincide;
\item $\gamma$ is not contractible into $ M$ or into the boundary of $ S$;
\item $\gamma$ does not cut out an unpunctured monogon or an unpunctured digon.
\end{itemize}

Two arcs are \textit{compatible} if there are arcs in their
respective isotopy  classes whose relative interiors do not
intersect.

An arc whose endpoints coincide is called a \textit{loop}.
\end{definition}

\begin{definition}\cite[Definition 2.6]{FST08}
An \textit{ideal triangulation} of $(S, M)$ is any
maximal  collection of pairwise compatible arcs whose relative
interiors do not intersect each other.

The arcs of the triangulation cut the surface $S$  into
\textit{ideal triangles}. The three sides of an ideal triangle do
not have to be distinct, i.e., we allow \textit{self-folded}
triangles.
\end{definition}

\begin{figure}[ht]
\centering
\begin{tikzpicture}
\filldraw [black] (0,0.9) circle (1pt)
                 (0,0) circle(1pt);
\draw  (0,0) to [out=45,in=350](0,1.5)
         (0,1.5) to [out=195,in=135](0,0);
\draw  (0,0) -- (0,0.9);
\draw (-.15, .5)node{$\gamma$};
\end{tikzpicture}
\caption{Self-folded ideal triangle}
\end{figure}
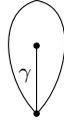

An easy calculation shows that any ideal triangulation of  a sphere with $n$ punctured consists of $3n$ arcs.


\section{Jacobian algebras arising from a sphere with $n$-punctures}\label{3}

The algebra arising from a sphere with punctures was studied for first time by Barot and Geiss in \cite{BG09}. They prove that the tubular
cluster algebra of type (2,2,2,2) corresponds to a sphere with
4-punctures (see definition of a cluster tilted algebra in Section \ref{4}). In this section, we study the Jacobian algebras arising from a sphere with
$n$-punctures, where $n\geq 5$.

It is well known that  finite-dimensionality of Jacobian algebra is preserved by mutations. Thus, in order to prove the Theorem \ref{TeoSphere}, it is
enough to prove that there exists an ideal triangulation $\TT$ such
that the Jacobian algebra  $\mathcal P(Q(\TT),W(\TT))$ is finite
dimensional. For that reason, we give a particular ideal triangulation with that property.

Consider the ideal triangulation $\TT$ and the quiver $Q(\TT)$ depicted in Figure \ref{carcaj}. For notational convenience we  label the punctures on the north and south poles with $p_{n+1}$ and $p_{n+2}$, so we will consider the sphere with $(n+2)$-punctures, where $n\geq 3$. The labels we have assigned to the arrows in Figure \ref{carcaj} will be kept throughout the paper.

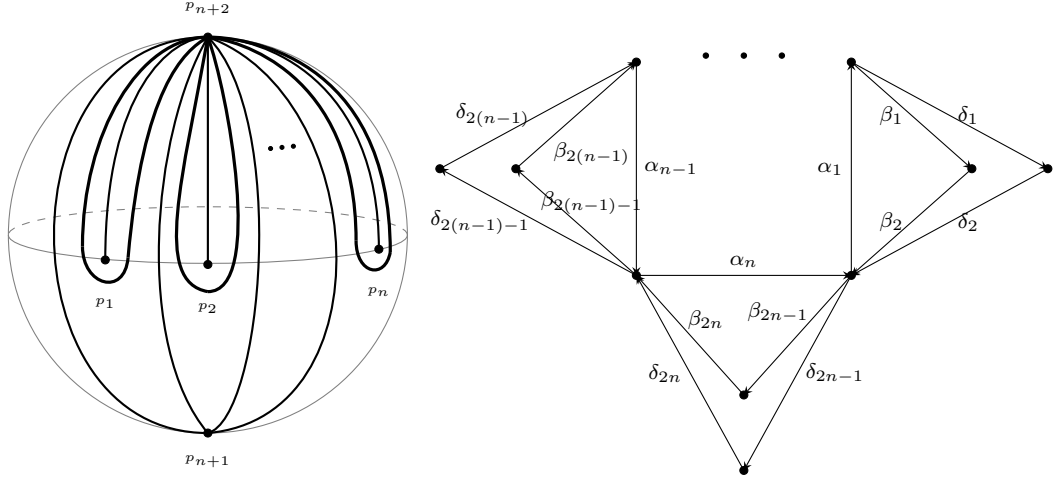
\begin{figure}[ht!]
\centering
\subfloat{
  \begin{tikzpicture}[scale=.75]
    \draw[gray] (3.5,3.5) circle (3.5cm);
    \draw[gray](0,3.5) arc (180:360: 3.5cm and 0.5cm);
    \draw[gray,thin,dashed] (7,3.5) arc (0:180: 3.5cm and 0.5cm);
    \draw[thick, black] (3.5,7) .. controls +(180:3.6cm) and +(180:3.6cm) .. (3.5,0);
    \draw[thick, black] (3.5,7) .. controls +(180:.65cm) and +(90:3cm) .. (1.7,3.06);
    \draw[thick, black] (3.5,7) .. controls +(0:1cm) and +(90:3cm) .. (6.5,3.25);
    \draw[thick, black] (3.5,7) .. controls +(-90:1cm) and +(90:2cm) .. (3.5,3.01);
    \draw[thick, black] (3.5,7) .. controls +(-20:3cm) and +(5:3.2cm) .. (3.5,0);
  \draw[thick, black] (3.5,7) .. controls +(-120:2cm) and +(130:2cm) .. (3.5,0);
  \draw[thick, black] (3.5,7) .. controls +(-45:2cm) and +(15:1cm) .. (3.5,0);
    \draw[very thick, black] (3.5,7) .. controls +(0:.6cm) and +(90:3.5cm).. (6.7,3.2);
    \draw[very thick, black] (6.7,3.2) .. controls +(260:.5cm) and +(-90:.6cm).. (6.1,3.4);
    \draw[very thick, black] (6.1,3.4) .. controls +(90:3cm) and +(-5:.5cm).. (3.5,7);

    \draw[very thick, black] (3.5,7) .. controls +(180:1.2cm) and +(90:2cm).. (1.3,3.3);
    \draw[very thick, black] (1.3,3.3) .. controls +(270:.7cm) and +(-95:.7cm).. (2.1,3.1);
    \draw[very thick, black] (2.1,3.1) .. controls +(85:3cm) and +(-150:.5cm).. (3.5,7);
    %
    \draw[very thick, black] (3.5,7) .. controls +(-70:1cm) and +(0:1cm).. (3.5,2.5);
    \draw[very thick, black] (3.5,7) .. controls +(-100:2.5cm) and +(170:1cm).. (3.5,2.5);
    \filldraw [black] (3.5,7) circle (2pt)
                  (3.5,0) circle (2pt)
                  (1.7,3.06) circle(2pt)
                  (6.5,3.25) circle(2pt)
                  (3.5,2.98) circle(2pt);
\filldraw(4.6,5.03)circle(1pt)
        (4.8,5.05)circle(1pt)
        (5,5.07)circle(1pt);
\draw (3.5,7.5) node {\tiny$p_{n+2}$};
\draw (3.5,-.5) node {\tiny$p_{n+1}$};
\draw (1.7,2.3) node {\tiny$p_{1}$};
\draw (3.5,2.2) node {\tiny$p_{2}$};
\draw (6.5,2.5) node {\tiny$p_{n}$};
    \end{tikzpicture}
}
\subfloat{
\begin{tikzpicture}
  \newdimen\R
  \R=2cm
\draw[fill] (45:\R) circle (1.5pt)
            (135:\R) circle (1.5pt)
            (225:\R) circle (1.5pt)
            (315:\R) circle (1.5pt)
            (0:3) circle (1.5pt)
            (0:4) circle (1.5pt)
            (180:3) circle (1.5pt)
            (180:4) circle (1.5pt)
            (270:3) circle (1.5pt)
            (270:4) circle (1.5pt);
\path[-stealth]
(315:\R)edge node[left]{$\alpha_1$} (45:\R)
(135:\R)edge node[right]{$\alpha_{n-1}$} (225:\R)
(225:\R)edge node[above]{$\alpha_{n}$} (315:\R)
(45:\R)edge node[left]{\small$\beta_1$} (0:3)
(45:\R)edge node[right]{\small$\delta_1$} (0:4)
(0:3)edge node[left]{\small$\beta_2$} (315:\R)
(0:4)edge node[right]{\small$\delta_2$} (315:\R)
(315:\R)edge  (270:3)
(315:\R)edge node[right]{\small$\delta_{2n-1}$} (270:4)
(270:3)edge  (225:\R)
(270:4)edge node[left]{\small$\delta_{2n}$} (225:\R)
(225:\R)edge (180:3)
(225:\R)edge node[left]{\small$\delta_{2(n-1)-1}$} (180:4)
(180:3)edge  (135:\R)
(180:4)edge node[left]{\small$\delta_{2(n-1)}$} (135:\R);
\draw(-2,-0.45) node{\small$\beta_{2(n-1)-1}$};
\draw(-2,0.2) node {$\beta_{2(n-1)}$};
\draw(-0.5,-2) node{$\beta_{2n}$};
\draw(0.45,-1.9) node{$\beta_{2n-1}$};
\filldraw(-0.5,1.5)circle(1pt)
        (0,1.5)circle(1pt)
        (0.5,1.5)circle(1pt);
\end{tikzpicture}
}
\caption{Quiver associated with the ideal triangulation $\TT$}
\label{carcaj}
\end{figure}

For each puncture $p_i\in P$ in the sphere, we choose a non-zero scalar
$x_i\in k$. By \cite[Definition 23]{LF09}, the potential
$W(\TT)$ associated with the ideal triangulation $\TT$ and  according to the label in the arrows of the
quiver in Figure \ref{carcaj} is:

\begin{eqnarray*}
&W(\TT)=x_{n+1}\alpha_1\dots\alpha_n +x_{n+2}\delta_{2n-1}\delta_{2n}\dots\delta_1\delta_2\\
&+ \sum_{i=1}^{n}\alpha_{i}\beta_{2i-1}\beta_{2i}+\sum_{i=1}^n(-{x_i}^{-1})\alpha_i\delta_{2i-1}\delta_{2i}
\end{eqnarray*}

Before we prove Theorem \ref{TeoSphere}, we establish some useful identities in the Jacobian
algebra $\cP P(Q(\TT),W(\TT))$.

\begin{lemma}\label{cases}
The following identities hold in the Jacobian algebra $\cP(Q(\TT),W(\TT))$:
\begin{eqnarray}
\beta_{2i-1}\beta_{2i}\delta_{2(i-1)-1}\delta_{2(i-1)}&=&x_{i-1}\beta_{2i-1}\beta_{2(i-1)}\beta_{2(i-1)-1}\beta_{2(i-1)}\label{ecua1}\\
&=&(x_{i-1}/x_{i})\delta_{2i-1}\delta_{2i}\beta_{2(i-1)-1}\beta_{2(i-1)}\label{ecua2}\\
&=&x_{i}^{-1}\delta_{2i-1}\delta_{2i}\delta_{2(i-1)-1}\delta_{2(i-1)}\label{ecua3}
\end{eqnarray}
for every $i=1, \dots, n$
\end{lemma}
\begin{proof}

Let $\Lambda$ be the Jacobian algebra $\cP(Q(\TT), W(\TT))$.
Since $$\partial_{\alpha_{i-1}}(W(\TT))=\beta_{2(i-1)-1}\beta_{2(i-1)}+x_{n+1}\alpha_{i}\dots\alpha_{i-2}-x_{i-1}^{-1}\delta_{2(i-1)-1}\delta_{2(i-1)},$$
 then in $\Lambda$ we have the identity
$$\beta_{2i-1}\beta_{2i}\delta_{2(i-1)-1}\delta_{2(i-1)}=x_{n+1}x_{i-1}\beta_{2i-1}\beta_{2i}\alpha_{i}\dots\alpha_{i-2} $$ $$+x_{i-1}\beta_{2i-1}\beta_{2i}\beta_{2(i-1)-1}\beta_{2(i-1)}.$$

Observe that $\partial_{\beta_{2i-1}}(W(\TT))= \beta_{2i}\alpha_{i}$, then the first term on the right hand is in the Jacobian ideal, therefore

$$\beta_{2i-1}\beta_{2i}\delta_{2(i-1)-1}\delta_{2(i-1)}=
x_{i-1}\beta_{2i-1}\beta_{2i}\beta_{2(i-1)}\beta_{2(i-1)-1}\beta_{2(i-1)}$$


This establishes the first identity. The second identity can be
proved in a similar fashion and is left to the reader. Let us show
the third identity. By the relations induced by
$\partial_{\alpha_{i-1}}(W(\TT))$, we have:
$$\beta_{2i-1}\beta_{2i}\delta_{2(i-1)-1}\delta_{2(i-1)}=-x_{n+1}\alpha_{i+1}\dots\alpha_{i-1}\delta_{2(i-1)-1}\delta_{2(i-1)}$$
$$+x_{i-1}^{-1}\delta_{2i-1}\delta_{2i}\delta_{2(i-1)-1}\delta_{2(i-1)}$$

Denote by  $\rho$ the term
$\alpha_{i+1}\dots\alpha_{i-2}\alpha_{i-1}\delta_{2(i-1)-1}\delta_{2(i-1)}$.
Notice that $\rho$ is a path of length $n+1$.  We claim that $\rho$
is in the Jacobian ideal.

Using $\partial_{\delta_{2(i-1)}}(W(\TT))$, we have the following identity $$\alpha_{i-1}\delta_{2(i-1)-1}={x_{i-1}}x_{n+2}\delta_{2(i-2)-1}
\delta_{2(i-2)}\dots\delta_1\delta_2$$.
Then, replacing $\alpha_{i-1}\delta_{2(i-1)-1}$ in $\rho$, we have
$$\rho=x_{n+2}x_{i-1}\alpha_{i+1}\dots\alpha_{i-2}\delta_{2(i-2)-1}
\delta_{2(i-2)}\dots\delta_1\delta_2 \delta_{2(i-1)}.$$
Observe that $\rho$ is a path of length $3n$ in $\Lambda$.

Replacing $\delta_{2(i-2)-1}\delta_{2(i-2)}$ by the relation induced
by $\partial_{\alpha_{i-2}}(W(\TT))$, the path $\rho$ is a path of
length $4n-3$. Iterating this process and using the topology of the
Jacobian algebra (see  Remark \ref{topology}), we have that $\rho$
is in the Jacobian ideal.

Then, $\beta_{2i-1}\beta_{2i}\delta_{2(i-1)-1}\delta_{2(i-1)}=x_{i-1}^{-1}\delta_{2i-1}\delta_{2i}\delta_{2(i-1)-1}\delta_{2(i-1)}$.
\end{proof}

\begin{lemma}\label{cases2}
The following identities hold in the Jacobian algebra $\mathcal P(Q(\TT),W(\TT))$.
\begin{eqnarray}
\alpha_i\delta_{2i-1}\delta_{2i}&=&x_{i}x_{i-1}\delta_{2(i-1)-1}\delta_{2(i-1)}\alpha_{i-1}\\
 \alpha_{i}\alpha_{i+1}\delta_{2(i+1)-1}&=&\delta_{2(i-1)}\alpha_{i-1}\alpha_{i}=0
\end{eqnarray}
for every $i=1, \dots, n$
\end{lemma}

\begin{proof}
The identity (4) follows as the first two identities in Lemma \ref{cases}. We proof the second one.

Notice that
$$\partial_{\delta_{2(i+1)}}(W(\TT))=-x_{i+1}^{-1}\alpha_{(i+1)}\delta_{2(i+1)-1}
+x_{n+2}\delta_{2(i)-1}\delta_{2i}\dots\delta_{2(i+1)-1},$$ then we have the identity

$$\alpha_{i}\alpha_{i+1}\delta_{2(i+1)-1}=
x_{i+1}x_{n+2}\alpha_{i}\delta_{2i-1}\delta_{2i}\dots\delta_{2(i+1)-1}\delta_{2(i+1)}.$$

Let $\rho$ be the path
$\alpha_{i}\delta_{2i-1}\delta_{2i}\dots\delta_{2(i+1)-1}\delta_{2(i+1)}$
By the identity (3) in Lemma \ref{cases} we have that $\rho$ is
equal to

$$\alpha_{i}
\beta_{2(i)-1}\beta_{2i}\delta_{2(i-1)-1}\dots\delta_{2(i+1)},$$
which is in the Jacobian ideal because it contains a factor $\alpha_1\be_{2i-1}=\partial_{\be_{2i}}(W(\TT))$.

Then $\alpha_{i}\alpha_{i+1}\delta_{2(i+1)-1}=0$ in $\cP(Q(\TT),W(\TT))$.
\end{proof}

\begin{lemma}\label{camino}
Let $\rho$ be a non zero path of length $5$ that starts in $\al_i$
or $\delta_{2i-1}$. If $\rho$ is not involving any arrow $\beta_i$
for $i=1, \dots, 2n$, then $rho$ is either the path
$\al_i\dots\al_{i+4}$ or
$\delta_{2i-1}\delta_{2i}\delta_{2(i-1)-1}\delta_{2(i-1)}$ in
$\cP(Q(\TT), W(\TT))$.
\end{lemma}

\begin{proof}
First suppose $\rho$  starts with $\delta_{2i-1}$. Then $\rho$ is one of the following:

\begin{itemize}
\item $\delta_{2i-1}\delta_{2i}\delta_{2(i-1)-1}\delta_{2(i-1)}\delta_{2(i-2)-1}$ or
 \item
$\delta_{2i-1}\delta_{2i}\al_{i}\delta_{2i-1}\delta_{2i}=x_ix_{i-1} \delta_{2i-1}\delta_{2i}\delta_{2(i-1)-1}\delta_{2(i-1)}\al_{i-1}$ or
\item $\delta_{2i-1}\delta_{2i}\al_{i}\alpha_{i+1}x$, where $x=\alpha_{i+2}$ o $\delta_{2(i+1)-1}$
\end{itemize}

By Lemma \ref{cases}, the second option is zero, and by Lemma
\ref{cases2}, the third option is zero. Then
$\rho=\delta_{2i-1}\delta_{2i}\delta_{2(i-1)-1}\delta_{2(i-1)}\delta_{2(i-2)-1}$

Now suppose $\rho$ starts with $\al_i$, then $\rho$ is one of the following:

\begin{itemize}
\item $\al_{i}\delta_{2i-1}\delta_{2i}\al_{i}\delta_{2i-1}$ or
\item $\al_{i}\delta_{2i-1}\delta_{2i}\delta_{2(i-1)-1}\delta_{2(i-1)}$
\item $\al_i\dots\al_{i+4}$
\end{itemize}

By the first part of  Lemma \ref{cases2} the first one is equal to
$$x_{i}x_{i-1}\delta_{2(i-1)-1}\delta_{2(i-1)}\alpha_{i-1}\al_{i}\delta_{2i-1},$$
and by the second part of the same Lemma, that path is zero.
Finally, the path
$\al_{i}\delta_{2i-1}\delta_{2i}\delta_{2(i-1)-1}\delta_{2(i-1)}$ by
Lemma \ref{cases} is  equal to
$$x_i\al_i\beta_{2i-1}\beta_{2i}\delta_{2(i-1)-1}\delta_{2(i-1)},$$
there force this path is also zero. Then $\rho=\al_i\dots\al_{i+4}$
\end{proof}


\begin{remark}\label{resumen}
Notice that, since
$\alpha_i\beta_{2i-1}=\partial_{\beta_{2i}}(W(\TT))$ and
$\be_{2i}\al_i=\partial_{\be_{2i-1}}(W(\TT))$, every path containing
either of the path $\alpha_i\beta_{2i-1}$ or $\be_{2i}\al_i$ as a
factor  is zero in $\cP(Q(\TT),W(\TT))$. Moreover, every path $\rho$
such that there exists at least a $\alpha_i$ and at least a
$\beta_j$, for $1\leq i \leq n$ and $1\leq j\leq 2n$, is zero. The
last assertion follows from applying repeatedly the first part of
Lemma \ref{cases2} until we obtain a factor $\alpha_i\beta_{2i-1}$
or $\be_{2i}\al_i$.
\end{remark}

Now, we can prove our main result.

\begin{proof}[Proof of Theorem \ref{TeoSphere}]
Since finite-dimensionality of Jacobian algebras is invariant under
mutations (cf. \cite[Corollary 6.6]{DWZ08}) and flips of ideal
triangulations are compatible with mutations of quivers with
potentials (cf. \cite[Theorem 30]{LF09}), it is enough to show that
$\cP(Q(\TT),W(\TT))$ is finite dimensional, where $\TT$ is the
triangulation in Figure \ref{carcaj}. We shall prove that every path
of length at least $2n+2$ belongs to the Jacobian ideal $J(W(\TT))$.

Let $\rho$ be a path of length at least $2n+2$. Without loss of
generality we can assume that $\rho$ starts with $\beta_{2n-1}$ or
$\delta_{2n-1}$. Denote by $Q_\rho$ the set of arrows of the path
$\rho$. By Remark \ref{resumen}, it is enough to analyze when
$Q_\rho \subset Q(\TT)_1\setminus\{\al_1, \dots, \al_n\}$ or
$Q_\rho\subset Q(\TT)_1\setminus\{\be_1, \dots, \be_{2n}\}$.

Consider the first case. Without loss of generality we can assume
that $\rho$ starts with $\beta_{2n-1}$ or $\delta_{2n-1}$. Then by
Lemma \ref{cases}, $$\rho=x\delta_{2n-1}\delta_{2n}\dots
\delta_1\delta_2\delta_{2n-1}\delta_{2n}\rho',$$  where $\rho'$ is
the rest of the path $\rho$ and $x\in k$ is the product certain
scalars $x_{j}$, because we can always change a factor
$\be_{2i-1}\be_{2i}\delta_{2(i-1)-1}\delta_{2(i-1)}$ or
$\be_{2i-1}\be_{2i}\be_{2(i-1)-1}\be_{2(i-1)}$ or
$\delta_{2i-1}\delta_{2i}\be_{2(i-1)-1}\be_{2(i-1)}$ by
$\delta_{2i-1}\delta_{2i}\delta_{2(i-1)-1}\delta_{2(i-1)}$.

Hence, by the relation induced of the partial derivative
$$\partial_{\delta_2}(W(\TT))=x_{n+2}\delta_{2n-1}\delta_{2n}\dots\delta_1 +x_1\alpha_1\delta_1,$$
we have that $$\rho=-xx_1\alpha_1\delta_1\delta_2\delta_{2n-1}\delta_{2n}\rho'.$$

Then  by Lemma \ref{cases} the path $\rho$ is zero in $\cP(Q(\TT),W(\TT))$.

Now suppose $Q_\rho\subset Q(\TT)_1\setminus\{\be_1, \dots,
\be_{2n}\}$. By Lemma \ref{camino}, the only non zero factor of
length 5 is $\al_i\dots\al_{i+4}$ or
$\delta_{2i-1}\dots\delta_{2(i-2)-1}$. Then we can assume that
$\rho=\delta_{2n-1}\delta_{2n}\dots
\delta_1\delta_2\delta_{2n-1}\delta_{2n}\rho'$ or
$\rho=\al_1\dots\al_n\al_1\dots\al_n\rho'$. But we have already
proof that the first option is a zero path. The second one using the
relation induced by $\partial_{\al_1}(W(\TT))$, we have:
$$\rho=x_{n+1}^2x_{1}^2\al_1\delta_1\delta_2\al_1\delta_1\delta_2\rho'$$
 that it is zero by Lemma \ref{cases2}.
\end{proof}

\section{Potentials in a class of cyclically oriented Quivers}\label{4}

In this section, we construct finite dimensional Jacobian algebras from quivers with certain combinatorial characteristics.

First we recall the definition of primitive potential (c.f \cite{DWZ08}, Section 9) and cyclically oriented quivers.

\begin{definition}(\cite{BT10}, Definition 3.1)
A \emph{walk of length $p$} in a quiver $Q$ is a $(2p+1)$-tuple
$$
w=(x_p,\alpha_p,x_{p-1},\alpha_{p-1},\ldots,x_1,\alpha_1,x_0)
$$
such that for all $i$ we have $x_i\in Q_0$, $\alpha\in Q_1$ and
$\{s(\alpha_i),e(\alpha_i)\}=\{x_p,x_{p-1}\}$. The walk $w$ is
\emph{oriented} if either $s(\alpha_i)=x_{p-1}$ and
$e(\alpha_i)=x_p$ for all $i$ or $s(\alpha_i)=x_{p}$ and
$e(\alpha_i)=x_{p-1}$ for all $i$. Furthermore, $w$ is called a
\emph{cycle} if $x_0=x_p$. A cycle of length $1$ is called a
\emph{loop}. We often omit the vertices and abbreviate $w$ by
$\alpha_p\cdots\alpha_1$. An oriented walk is also called
\emph{path}.

A cycle $c=(x_p,\alpha_p,\ldots,x_1,\alpha_1,x_p)$ is called
\emph{non-intersecting} if its vertices $x_1,\ldots,x_p$ are
pairwise distinct. A non-intersecting cycle of length $2$ is called
$2$-cycle. If $c$ is a non-intersecting cycle then any arrow
$\beta\in Q\setminus\{\alpha_1,\ldots,\alpha_p\}$ with
$\{s(\beta),e(\beta)\}\subseteq\{x_1,\ldots,x_p\}$ is called a
\emph{chord} of $c$. A cycle $c$ is called \emph{chordless} if it is
non-intersecting and there is no chord of $c$.

A quiver $Q$ without loop and $2$-cycle is call \emph{cyclically
oriented} if each chordless cycle is oriented. Note that this
implies that there are no multiple arrows in $Q$. A quiver without
oriented cycle is called \emph{acyclic} and an algebra whose quiver
is acyclic is called \emph{triangular}.
\end{definition}

\begin{definition}
Let $Q$ be a quiver. A \emph{primitive potential} $S$ is a lineal
combination of every oriented chordless cycle in $Q$ with non-zero
scalars.
\end{definition}

Buan, Marsh, Reineke,  Rieten and Todorov introduced in
\cite{BMRRT06} a cluster category $\cC_A$ associated to a hereditary
algebra $A$ and  proved that $\cC_A$ is endowed with a
cluster-tilting object. The endomorphism algebra of a
cluster-tilting object is called \emph{cluster tilted algebra} and
it was proven in \cite{Ke2011} by Keller that any cluster tilted
algebra is a Jacobian algebra.

Barot and Trepode in \cite{BT10} given an explicit description of the minimal
relations in cluster tilted algebras with cyclically oriented quivers, and
it follows from this result that the potential associated with this
kind of algebras is primitive.

Amiot in \cite{Ami09} introduced a cluster category $\cC_{(Q,W)}$
associated to a quiver with potential $(Q,W)$, and she proved that
when the Jacobian algebra $\cP(Q,W)$ is finite dimensional, the
category $\cC_{(Q,W)}$ is endowed with a cluster-tilting object
whose endomorphism algebra is isomorphic to $\cP(Q,W)$. In these
context, the endomorphism algebra of a cluster-tilting object is
called \emph{2-Calabi-Yau tilted algebra}.

Observe that the quiver $Q(\TT)$ in Figure \ref{carcaj} is
cyclically oriented however the potential $W(\TT)$ is not primitive,
showing that the previous result does not extend to Jacobian
algebras. Moreover it is easy to proof that  the Jacobian Algebra
$\mathcal P(Q(\TT),W)$ is not finite dimensional for any $W$ a
primitive potential.

\begin{proposition}
Let $W=\sum_{i=1}^n
y_i\al_i\be_{2i-1}\be_{2i}+\sum_{i=1}^nz_i\al_i\gamma_{2i-1}\gamma_{2i}+
w_i\al_1\dots\al_n$ be a primitive potential, where $y_i, z_i,
w_i\in k$. Then the Jacobian algebra $\mathcal P(Q(\TT), W)$ is not
finite dimensional.
\end{proposition}

\begin{proof}
As we mention before, finite-dimensionality of Jacobian algebras is
invariant under mutations (cf. \cite[Corollary 6.6]{DWZ08}). Then it
is enough to find a  mutation $\mu_\alpha$ of the quiver with a
primitive potential $(Q(\TT),W)$ such that $\mathcal
P(\mu_\alpha(Q(\TT),W))$ is not a finite dimensional algebra.

For instance, consider the mutation $\mu_{\alpha_1}(Q(\TT),W)=
(\mu_{\alpha_1}Q(\TT), \mu_{\alpha_1} W)$.  An easy calculation show that
$\mu_{\alpha_1} W= \sum_{i=2}^ny_i\al_i\be_{2i-1}\be_{2i}+\sum_{i=2}^n
z_i\al_i\gamma_{2i-1}\gamma_{2i}$ and $\mu_a(Q)$ has the quiver of
Figure \ref{carcaj2}.
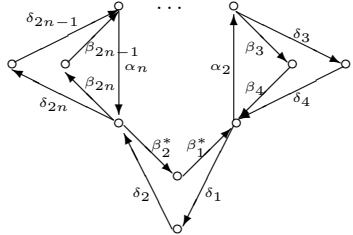
\begin{figure}[ht!]
\begin{center}
\begin{picture}(125,86)
\put(40,40){\circle{3}}
\put(84,40){\circle{3}}
\multiput(40,84)(43,0){2}{\circle{3}}
\put(0,62){\circle{3}}
\multiput(20,62)(85,0){2}{\circle{3}}
\put(125,62){\circle{3}}
\multiput(62,0)(0,20){2}{\circle{3}}
\put(83,41){\vector(0,1){40}}
\put(84,82){\vector(2,-1){39}}
\put(84,82){\vector(1,-1){19}}
\put(124,61){\vector(-2,-1){39}}
\put(104,61){\vector(-1,-1){19}}
\put(64,20){\vector(1,1){18}}
\put(83,39){\vector(-1,-2){19}}
\put(42,39){\vector(1,-1){18}}
\put(60,0){\vector(-1,2){18}}
\put(40,82){\vector(0,-1){39}}
\put(38,41){\vector(-1,1){18}}
\put(38,41){\vector(-2,1){38}}
\put(1,63){\vector(2,1){39}}
\put(21,63){\vector(1,1){19}}
\put(65,30){\tiny$\beta^*_1$}
\put(52,30){\tiny$\beta^*_2$}
\put(74,60){\tiny$\alpha_2$}
\put(42,60){\tiny$\alpha_n$}
\put(87,67){\tiny$\beta_3$}
\put(87,52){\tiny$\beta_4$}
\put(27,67){\tiny$\beta_{2n-1}$}
\put(27,54){\tiny$\beta_{2n}$}
\put(72,12){\tiny$\delta_1$}
\put(45,12){\tiny$\delta_2$}
\put(105,48){\tiny$\delta_4$}
\put(105,72){\tiny$\delta_3$}
\put(10,46){\tiny$\delta_{2n}$}
\put(5,78){\tiny$\delta_{2n-1}$}
\put(54,83){$\dots$}
\end{picture}
\end{center}
\caption{Quiver of $\mu_{\alpha_1}(Q(\TT),W)$}
\label{carcaj2}
\end{figure}

Notice that the arrows $\be^*_2,\be^*_1,\delta_1,\delta_2$ are not
in the Jacobian ideal, then the cycle
$(\be^*_2\be^*_1\delta_1\delta_2)^j$ is not zero for any $j\in
\mathbb N$, then  the Jacobian algebra $\mathcal
P(\mu_{\alpha_1}(Q(\TT),W))$ is not finite dimensional, therefore
$\mathcal P(Q(\TT),W)$ either.
\end{proof}

\begin{definition}(\cite{BT10}, Definition 3.3)
 A path $\gamma$ which is anti-parallel to an arrow $\eta$ in a quiver $Q$ is a \emph{shortest path} if
the full subquiver generated by the induced oriented cycle $\eta\gamma$ is chordless.
A path $\gamma=(x_0\xrightarrow{\gamma_1} x_1 \xrightarrow x_2 \rightarrow \cdots
   \rightarrow x_L)$ is called \emph{shortest directed path} if there exists no arrow $x_i\rightarrow x_j$ in $Q$ with $1\leq i+1<j\leq L$.
A walk $\gamma=(x_0\myjoin x_1 \myjoin x_2 \myjoin \cdots \myjoin
x_L)$ is called a \emph{shortest walk} if there is no edge joining
$x_i$ with $x_j$ with $1\leq i+1<j\leq L$ and $(i,j)\neq(0,L)$ (we
write a horizontal line to indicate an arrow oriented in one of the
two possible ways).
\end{definition}

\begin{definition}
Let $Q$ be a cyclically oriented quiver such that for any arrow
$\alpha$ there are at most 2 anti-parallel shortest path to $\alpha$
and $c=\beta_0\beta_1\dots\beta_L$ a oriented chordless cyclic. We
construct a sequence of triples $(\alpha_n, \rho_n, \rho_n')\in
Q_0\times (Q_1\cup\{0\})\times Q_1$ for each $n\in \mathbb N\cup
\{0\}$, such that $\rho_n$ and $\rho_n'$ are anti-parallel shortest
paths to $\alpha_n$ or $\rho_n=0$ and $\rho_n'$ is the anti-parallel
shortest path to $\alpha_n$ if $\alpha_n$ has only one anti-parallel
shortest path, in the following way:

\begin{itemize}
\item[\textbf{Step 0}] We denote by $\al_0$ to the arrow $\beta_0$ and  by $\rho_0'$ the anti-parallel shortest path $\beta_1\dots\beta_L$ to the arrow $\al_0$ in $c$. If there exists an anti-parallel shortest path $\rho_0$ to $\al_0$ different to $\rho$, then the first element of the sequence is $(\alpha_0, \rho_0, \rho_0')$. Otherwise, the sequence is constant to the element $(\al_0,0, \rho_0')$.

\item[\textbf{Step 1}] We denote by $\al_1$ the arrow in the path $\rho$ such that $t(\al_0)=s(\al_{1})$ and by $\rho'_1$ the anti-parallel shortest path to $\alpha_1$ in $\al_0\rho_0$. If there exists an anti-parallel shortest path $\rho_1$ to $\al_1$ different to $\rho'_1$, then the second element of the sequence is $(\alpha_1, \rho_1, \rho_1')$. Otherwise,  $(\al_n,\rho_n, \rho_n')=(\al_1,0, \rho_1')$ for each $n\geq 1$.

\item[\textbf{Step 2}] We denote by $\al_2$ the arrow in the path $\rho_1$ such that $s(\al_1)=t(\al_{2})$ and by $\rho'_2$ the anti-parallel shortest path to $\alpha_2$ in $\alpha_1\rho_1$. If there exists an anti-parallel shortest path $\rho_2$ to $\al_2$ different to $\rho'_2$, then the third element of the sequence is $(\alpha_2, \rho_2, \rho_2')$. Otherwise,  $(\al_n,\rho_n, \rho_n')=(\al_2,0, \rho_2')$ for each $n\geq 2$.

\item[] \hspace{4cm} $\vdots$
\item[\textbf{Step i}] We denote by $\al_i$ the arrow in the path $\rho_{i-1}$ such that
    \begin{itemize}
    \item $t(\be_i)=s(\be_{i+1})$ if $i$ is even or;
    \item $s(\be_i)=t(\be_{i+1})$ if $i$ is odd.
    \end{itemize}
and by $\rho'_i$ the anti-parallel shortest path to $\alpha_i$ in
$\alpha_{i-1}\rho_{i-1}$. If there exists an anti-parallel shortest
path $\rho_i$ to $\al_i$ different to $\rho'_i$, then the element
$i+1$ of the sequence is $(\alpha_i, \rho_i, \rho_i')$. Otherwise,
$(\al_n,\rho_n, \rho_n)=(\al_{i},0,  \rho_{i}')$ for each $n\geq i$.
\end{itemize}
The sequence $\{(\al_n, \rho_n, \rho_n')\}_{n\in \mathbb N\cup
\{0\}}$ is called \emph{cyclic sequence} of $c$. We say that the
cyclic sequence $\{(\al_n, \rho_n, \rho_n')\}_{n\in \mathbb N\cup
\{0\}}$ is \emph{finite} if there exists $m\in\mathbb N$ such that
$(\alpha_n, \rho_n, \rho_n)=(\alpha_m,0, \rho_m')$ for every $n\geq
m$.
\end{definition}

\begin{remark}
Consider the quiver $Q$ in Figure \ref{Sphere4}, which is associated
to a triangulation of a sphere with 4 punctures. Observe that the
cyclic sequence of any oriented chordless cycle in $Q$ is infinite,
because there are exactly two anti-parallel paths to each arrow of
$Q$.

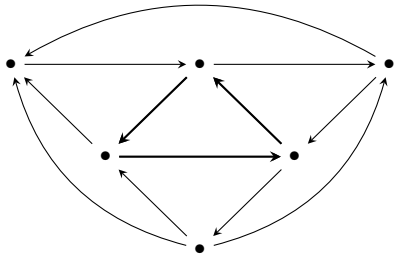
\begin{figure}[ht]
\begin{center}
\begin{tikzpicture}
\matrix (m)[matrix of math nodes, row sep=3em,column sep=3em]
    { \bullet& & \bullet& & \bullet\\
       &\bullet&  &\bullet&  \\
       & &\bullet & &\\};
\path[-stealth]
(m-1-1) edge (m-1-3)
(m-1-3) edge (m-1-5) edge  (m-2-2)
(m-2-2) edge (m-2-4) edge (m-1-1)
(m-2-4) edge (m-3-3) edge (m-1-3)
(m-1-5) edge (m-2-4)
(m-3-3) edge (m-2-2)
(m-1-5) edge[bend right] (m-1-1)
(m-3-3) edge[bend left] (m-1-1)
(m-3-3) edge[bend right] (m-1-5);
\path[-stealth, thick]
(m-1-3) edge (m-2-2)
(m-2-2) edge (m-2-4)
(m-2-4) edge (m-1-3);
\end{tikzpicture}
\end{center}
\caption{Quiver associated to an sphere with 4 punctures}
\label{Sphere4}
\end{figure}


\end{remark}

In the following theorem, we give combinatorial conditions of
the quiver such that the Jacobian algebra with a primitive potential is a finite dimensional algebra.

\begin{theorem}\label{TEO}
Let $Q$ be a cyclically oriented quiver such that:
\begin{enumerate}[i)]
\item for any arrow $\alpha$ there are at most 2 anti-parallel shortest path to $\alpha$;
\item the cyclic sequence of any oriented chordless cyclic $c$ is finite.
\end{enumerate}
If $W$ is a primitive potential of $Q$, then the Jacobian algebra $\mathcal P(Q,W)$ is a finite dimensional algebra.
\end{theorem}

\begin{proof}[Proof of Theorem \ref{TEO}]
Let $$W=\sum_{\stackrel{c}{\text{min. cycle}}} x_cc$$ be a primite
potential. It is enough to show that any non-intersecting oriented
cycle $c$ is zero in $\cP(Q,W)$. To fix notation denote by
$$c=(x_1\stackrel{\be_0}{\to}x_2\stackrel{\be_1}{\to}\dots x_{L-1}\stackrel{\be_{L}}{\to}x_1).$$

Suppose $c$ is chordless, then by hypothesis the cyclic sequence
$\{(\al_n, \rho_n, \rho_n')\}_{n\in \mathbb N\cup \{0\}}$ of $c$ is
finite. Let $m\in\mathbb N$ be the minimal number such that
$(\alpha_n,\rho_n, \rho_n')=(\alpha_m, 0,\rho_m')$ for every $n\geq
m$. Then we have
$$\partial_{\alpha_n}(W)=x_{c_n}\rho_n+x_{c_{n-1}}\rho_n'$$ for
every $n=1, \dots, m-1$ and
$$\partial_{\alpha_m}(W)=x_{c_{m-1}}\rho_m'$$
because there is only one anti-parallel shortest path to $\alpha_m$.

By construction of the cyclic sequence of $c$ we have that
$\rho_0'=\beta_1\ldots\beta_L$ and $\rho_0$ are anti-parallel
shortest paths to $\alpha_0$, then
$\beta_1\ldots\beta_L=-\frac{x_{c_0}}{x_{c}}\rho_0$ in $\cP(Q,W)$,
therefore,
\begin{eqnarray}\label{ciclo}
c &=& \alpha_0\beta_1 \dots \beta_L \nonumber \\
  &=& -\frac{x_{c_0}}{x_c}\alpha_0\rho_0
\end{eqnarray}
Recall that $\rho_1'$ is the anti-parallel shortest path of
$\alpha_1$ in the cycle $\alpha_0\rho_0$, then
$$c=-\frac{x_{c_0}}{x_c}\rho_1'\alpha_1.$$

Repeating this process for each triple of sequence we have

\begin{eqnarray}\label{ciclo1}
c &=& \left(-\frac{x_{c_0}}{x_c}\right)\left(-\frac{x_{c_1}}{x_{c_0}}\right)\rho_1\alpha_1= \frac{x_{c_1}}{x_{c}}\rho_1\alpha_1  \nonumber\\
&=& \frac{x_{c_1}}{x_{c}}\alpha_2\rho_2'\nonumber\\
&=& \left(\frac{x_{c_1}}{x_{c}}\right)\left(-\frac{x_{c_2}}{x_{c_1}}\right)\alpha_2\rho_2= -\frac{x_{c_2}}{x_{c}} \alpha_2\rho_2 \nonumber\\
&=& -\frac{x_{c_2}}{x_{c}}\rho_3'\alpha_3 \nonumber\\
&\vdots& \nonumber\\
&=& \begin{cases}
\frac{x_{c_{m-1}}}{x_{c}}\rho_{m-1}\al_{m-1} & \text{if $m-1$ is odd}\\
-\frac{x_{c_{m-1}}}{x_{c}}\alpha_{m-1}\rho_{m-1} & \text{if $m-1$ is even}
\end{cases}
\end{eqnarray}

Since $\rho_m'$ is the anti-parallel shortest path to $\alpha_m$ in
the cycle $\rho_{m-1}\alpha_{m-1}$ or $\alpha_{m-1}\rho_{m-1}$, the
expression \ref{ciclo1} can be rewritten in the following way:

\begin{equation}
c= \begin{cases}
\frac{x_{c_{m-1}}}{x_{c}}\alpha_{m}\rho'_{m} & \text{if $m-1$ is odd}\\
-\frac{x_{c_{m-1}}}{x_{c}}
\rho'_{m}\al_{m} & \text{if $m-1$ is even}
\end{cases}
\end{equation}

Then $c=0$ in $\cP(Q,W)$ because $\rho_l'$ is in the Jacobian ideal $J(W)$. Therefore any oriented chordless cycle is zero in $\cP(Q,W)$.

Suppose $c$ is non-chordless, then there exists a chord
$\be_1:x_i\to x_j$ with vertex in $c$ such that
$c_1=\gamma_1\be_{j+1} \dots\be_i$ is a oriented chordless cycle.
Consider $j$ the minimal number of the subset $\{0,1,2, \dots,
L-1\}$ such that  $c_1$ is a oriented chordless cycle and there is
not a chord in the path $\be_1 \ldots \be_j$. Denote by  $i_0=j$,
$i_1=i$ and by $\tilde{c_1}$ the walk which is obtained by replacing
the path $\be_{i_0+1} \dots \be_{i_1}$ by the arrow $\gamma_1$ in
the cycle $c$, namely,
$$\tilde{c_1}=(x_0\stackrel{\be_1}{\to}x_1\stackrel{\be_2}{\to}\dots
\stackrel{\be_{i_0}}{\to}x_{i_0}\stackrel{\mathbf{\gamma_1}}{\mathbf{\longleftarrow}}x_{i_1}\stackrel{\be_{i_1+1}}{\longrightarrow}\dots
x_{n-1}\stackrel{\be_{n}}{\to}x_0).$$

Since $\tilde{c_1}$ is a non-oriented cycle, then $\tilde{c_1}$ is
non-chordless, because $Q$ is a cyclically oriented quiver, then
there exists a chord $\gamma_2: x_{i_3}\to x_{i_2}$ such that
$c_2=\gamma_2\be_{i_2+1}\dots \be_{i_3}$ is oriented chordless cycle
and there is not a chord in the path $\beta_{i_1+1}\ldots
\beta_{i_2}$.

Let $\tilde{c_2}$ be the walk which is obtained by replacing the
path $\be_{i_2+1} \dots \be_{i_3}$ by the arrow $\gamma_2$ in the
walk $\tilde{c_1}$, namely,
$$\tilde{c_2}=(x_0\stackrel{\be_1}{\to}x_2\dots
\stackrel{\be_{i_0}}{\to}x_{i_0}\stackrel{\gamma_1}{\longleftarrow}x_{i_1}\stackrel{\be_{i_1+1}}{\longrightarrow}
\dots\stackrel{\be_{i_2}}{\to}x_{i_2}\stackrel{\gamma_2}{\longleftarrow}x_{i_3}\stackrel{\be_{i_3+1}}{\longrightarrow}
\dots x_{n-1}\stackrel{\be_{n}}{\to}x_1)$$ which is again not
oriented and therefore not chordless, then exists a chord
$\gamma_3:x_{i_5}\to x_{i_4}$, with the same properties of the
arrows $\gamma_1$ and $\gamma_2$, and a oriented chordless cycle
$c_3=\gamma_3\beta_{i_3+1}\ldots\beta_{i_4}$ and a walk
$\tilde{c_3}$.

Observe that the vertex of the arrows $\gamma_i$ are elements of an
increasingly smaller subset of $\{0, 1, \dots, L-1\}$, then there is
a natural number $r$ such that $\tilde{c_r}$ is oriented chordless
cycle and in particular $s(\gamma_i)=t(\gamma_{i+1})$ for every
$i=1, \dots, r$.

Then $$\partial_{\gamma_2}(W)= x_{c_1}\beta_{i_2+1} \dots
\beta_{i_3} + x_{c_2}\gamma_1\gamma_k\ldots\gamma_3$$ where $x_c,
x_c'\in k$, therefore $c$ can be rewritten as the following
\begin{equation}\label{ciclo2}-\frac{x_{c_2}}{x_{c_1}}\beta_1
\dots\be_{i_0}\gamma_1\gamma_k\ldots\gamma_3\beta_{i_3+1}\ldots\beta_{i_4}\ldots
\beta_n=-\frac{x_{c_2}}{x_{c_1}}\beta_1
\dots\be_{i_0}\gamma_1\gamma_k\ldots c_3\ldots \beta_n\end{equation}

Then $c$ is a zero path, because $c_3$ is oriented chordless cycle.

\end{proof}

\section*{Acknowledgments}

The second author thanks Professor Michael Barot for discussing some central ideas for this article. She thanks Professor Chrisof Geiss for pointing out some important results on surfaces with non-empty boundary.  She also thanks Daniel
Labardini-Fragoso for clarifying ideas of his article \cite{LF09}.
We thank Daniel Labardini-Fragoso for his helpful comments and
suggestions given in a preliminary version of this article. The author was partially supported by a CONICET doctoral fellowship.

\newcommand{\etalchar}[1]{$^{#1}$}


\end{document}